\documentclass{amsart}
\usepackage{amscd, amssymb, amsfonts}
\usepackage{amsmath,amsthm}
\usepackage{amssymb}

  [2002/01/19 v2.2g %
    AMS font definitions%
  ]
\DeclareFontFamily{U}{euf}{}
\DeclareFontShape{U}{euf}{m}{n}{%
  <5><6><7><8><9>gen*eufm%
  <10><10.95><12><14.4><17.28><20.74><24.88>eufm10%
  }{}
\DeclareFontShape{U}{euf}{b}{n}{%
  <5><6><7><8><9>gen*eufb%
  <10><10.95><12><14.4><17.28><20.74><24.88>eufb10%
  }{}

  [2002/01/19 v2.2g %
    AMS font definitions%
  ]
\DeclareFontFamily{U}{msb}{}
\DeclareFontShape{U}{msb}{m}{n}{%
  <5><6><7><8><9>gen*msbm%
  <10><10.95><12><14.4><17.28><20.74><24.88>msbm10%
  }{}

  [2002/01/19 v2.2g %
    AMS font definitions%
  ]
\DeclareFontFamily{U}{msa}{}
\DeclareFontShape{U}{msa}{m}{n}{%
  <5><6><7><8><9>gen*msam%
  <10><10.95><12><14.4><17.28><20.74><24.88>msam10%
  }{}

\newtheorem{theorem}{Theorem}[section]
\newtheorem{lemma}[theorem]{Lemma}
\newtheorem{proposition}[theorem]{Proposition}
\newtheorem{corollary}[theorem]{Corollary}

\theoremstyle{definition}

\theoremstyle{remark}
\newtheorem{remark}[theorem]{Remark}

\numberwithin{equation}{section}

\begin{document}

\title[On the behavior of $p$-adic Euler $\ell$-functions]
{On the behavior of $p$-adic Euler $\ell$-functions}

\author{Min-Soo Kim}

\address{Department of Mathematics, KAIST, 373-1 Guseong-dong, Yuseong-gu, Daejeon 305-701, South Korea}
\email{minsookim@kaist.ac.kr}



\begin{abstract}
In this paper we propose a construction of $p$-adic Euler $\ell$-function using Kubota-Leopoldt's approach
and Washington's one. We also compute the derivative of $p$-adic Euler $\ell$-function at $s=0$
and the values of $p$-adic Euler $\ell$-function at positives integers.
\end{abstract}

\maketitle

\def\C{\mathbb C_p}
\def\BZ{\mathbb Z}
\def\Z{\mathbb Z_p}
\def\Q{\mathbb Q_p}
\def\ord{\text{ord}_p}

\section{Introduction}\label{intro}

Let $p$ be an odd prime number.
Throughout this paper $\mathbb Z_p, \mathbb Q_p$ and $\mathbb C_p$
will denote the ring of $p$-adic integers, the field of $p$-adic numbers and the completion
of the algebraic closure of $\mathbb Q_p,$ respectively.
Let $v_p$ be the normalized exponential valuation of $\mathbb C_p$ with
$|p|_p=p^{-1}.$
$\mathbb Z, \mathbb Q, \mathbb R$ and $\mathbb C$ denote the ring of (rational) integers, the field of rational numbers,
the field of real numbers and the field of complex numbers, respectively.

The definition of Euler polynomials is well known and
appear in many classical results (cf. \cite{Ca0,MSK,TK1,KKJR,Ts,Yo}).
We consider the following generating function
\begin{equation}\label{def-E}
F(t,x)=\frac{2e^{xt}}{e^t+1}.
\end{equation}
By (\ref{def-E}) it is easy to get
$F(t,x+1)+F(t,x)=2e^{xt},F(t,-x)+F(-t,x)=2e^{-xt},$ and $F(t,1-x)-F(-t,x)=0.$
Expand $F(t,x)$ into a power series of $t:$
\begin{equation}\label{def-E-pow}
F(t,x)=\sum_{n=0}^\infty E_{n}(x)\frac{t^n}{n!}.
\end{equation}
The coefficients $E_n(x),n\geq0,$ are called Euler polynomials.
The Euler numbers $E_n$ can be readily from (\ref{def-E}) that
$E_n(0)=E_n$ (cf. \cite{TK1}).
It is easy to see that Euler polynomials $E_n(x)$ satisfy the identities
\begin{equation}\label{def-E-n-p}
\begin{aligned}
&E_n(x+1)+E_n(x)=2x^n; \\
&E_n(-x)+(-1)^{n}E_n(x)=2(-1)^nx^n; \\
&E_{n}(1-x)=(-1)^nE_n(x),
\end{aligned}
\end{equation}
where $n\geq0.$
Observe that $E_n+(-1)^nE_n=0,E_n=(-1)^nE_n(1),$ $E_n(-1/2)=(-1)^n2^{1-n}$ for $n\geq1,$
and $E_n(1/2)=0$ if $n$ is odd.
This yields the values $E_0=1,E_1=-1/2,E_3=1/4,E_5=-1/2,E_7=17/8,E_9=-31/2,E_{11}=691/4,\ldots,$ with $E_{2k}=0$ for $k=1,2,\ldots.$

For a primitive Dirichlet character $\chi$ with an odd conductor $f=f_\chi,$
the formal power series $F_\chi(t)$ are defined by
\begin{equation}\label{gen-E-numb}
F_\chi(t)=2\sum_{a=1}^{f}\frac{(-1)^a\chi(a)e^{at}}{e^{ft}+1},\quad |t|<\frac\pi f
\end{equation}
(see e.g. \cite{Ca0,Iwa,TK1,KKJR,Yo,Ts,Wa}).
Generalized Euler numbers $E_{n,\chi}$ belong to the Dirichlet character $\chi$ are defined by
\begin{equation}\label{gen-E-numb-d}
F_\chi(t)=\sum_{n=0}^{\infty}E_{n,\chi}\frac{t^n}{n!}.
\end{equation}
Let $\mathbb Q(\chi)$ denote the field generated over $\mathbb Q$ by all the values $\chi(a),a\in\mathbb Z.$
Then it can be shown that $E_{n,\chi}\in\mathbb Q(\chi)$ for each $n\geq0.$
In the complex case, the generating function $F_\chi(t)$ is given by
\begin{equation}\label{power-sum}
\begin{aligned}
F_\chi(t)=2\sum_{a=1}^{f}(-1)^a\chi(a)\sum_{k=0}^\infty(-1)^ke^{(a+fk)t}
=\sum_{n=0}^\infty\left(2\sum_{l=1}^\infty(-1)^l\chi(l)l^n\right)
\frac{t^n}{n!}.
\end{aligned}
\end{equation}
Of course, none of sums above are convergent and the argument is
not rigorous. However, the argument can be made rigorous in the
analytic continuation when the series (\ref{gen-E-numb}) and
(\ref{power-sum}) have a common domain of convergence complex plane (cf. \cite{Be}).
Comparing coefficients of ${t^n}/{n!}$ on both sides of (\ref{gen-E-numb-d}) and (\ref{power-sum}) gives
\begin{equation}\label{power-sum-ell}
E_{n,\chi}=2\sum_{l=1}^\infty(-1)^l\chi(l)l^n
\end{equation}
(see \cite[Theorem 7]{TK1}).

This is used to construct the Euler $\ell$-functions attached to $\chi.$

Let $\chi$ be a primitive Dirichlet character with an odd conductor $f.$
Define the Euler $\ell$-function attached to $\chi$ by
\begin{equation}\label{ell-ft}
\ell_E(s,\chi)=2\sum_{n=1}^\infty\frac{(-1)^n\chi(n)}{n^s},
\end{equation}
where Re$(s)>0$ (cf. \cite{TK1,TK2,KKJR}).
The Euler $\ell$-function attached to $\chi$ can be continued into the complex plane.
It is believed that the analysis of Euler $\ell$-function $\ell_E(s,\chi)$ attached to $\chi$ began with Euler's study of
the zeta function $\zeta(s)$
in which he considered the function only for real values of $s$ (cf. \cite{Ay}).
In particular, if we substitute $\chi=\chi^0,$ the trivial character, in (\ref{ell-ft}), we have put
\begin{equation}\label{E-zeta-ft}
\zeta_E(s)=2\sum_{n=1}^\infty\frac{(-1)^{n}}{n^s}
\end{equation}
(cf. \cite{Ay,TK1,TK2}).
The term $1-2^{1-s}$ is zero only at the isolated points $s_n=1+2\pi in/\log(2)$ for $n$ any integer. So, we
can use $\zeta_E(s)$ to define $\zeta(s)$ on this larger set by defining
\begin{equation}\label{zeta-zeta-E}
\zeta_E(s)=-2(1-2^{1-s})\zeta(s)\quad\text{for }\text{Re}(s)>0, \;s\neq s_n.
\end{equation}
Here, $\zeta(s)$ denotes the Riemann zeta function defined by
$\zeta(s)=\sum_{n=1}^\infty{n^{-s}}$
(see \cite{Ay,Ko,MR,Wa2}).
The function $\zeta_E(s)$ is called the Euler zeta function (cf. \cite{TK1,KKJR}).

Recall that the Dirichlet $L$-function attached to $\chi$ is defined by
\begin{equation}
L(s,\chi)=\sum_{n=1}^\infty\frac{\chi(n)}{n^s},
\end{equation}
for $s\in\mathbb C$ with $\text{Re} (s)>1.$ This function can be continued analytically to the entire complex plane,
except for a simple $s=1$ when $\chi=\chi^0,$
in which case we have the Riemann zeta function, $\zeta(s)=L(s,\chi^0).$
It is known that the values of $L(s,\chi)$ for $s\in\mathbb Z$ with $s\leq0$ are algebraic numbers so that they may be considered
as elements of the field $\overline{\mathbb Q}_p$ (see \cite{Iwa}).

The existence of the $p$-adic analogue of Dirichlet $L$-function was, in fact, proved
by Kubota-Leopoldt \cite{KL}, and they called it the $p$-adic $L$-function for the character $\chi.$
Nowadays, $p$-adic $L$-functions are quite important in number
theory in particular by the works of Iwasawa \cite{Iwa}, and different properties
of these functions have been investigated by various authors, and are strictly related,
by methods and applications, to many works appeared in the recent literature,
where different analytic properties of various type of $p$-adic $L$-functions are investigated
(see \cite{Di2,FG,Fo,Iwa,KS2,TK1,Ko,KL,Sh,Si06,Wa}).

Recently, $p$-adic interpolation functions of Euler numbers have been treated by Tsumura \cite{Ts},
Kim \cite{TK1}, Kim et al. \cite{KKJR}, Young \cite{Yo}.
In \cite{TK1}, Kim showed several properties and derived formulas involving the generalized Euler
numbers with a parameter.
Simsek \cite{Si06} studied twisted $(h,q)$-Bernoulli numbers and polynomials, and $(h,q)$-Euler numbers and polynomials.

The purpose of this paper is to give a construction of $p$-adic Euler $\ell$-function
which interpolates the values
$(1-\chi_n(p)p^n)\ell_E(-n,\chi_n)$
for $n\geq0$
using Kubota-Leopoldt's method \cite{KL} (see Section \ref{con-eu-ell} for definitions).
Let us recall that another construction was done in \cite{KKJR}.
We also compute the derivative of $p$-adic Euler $\ell$-function at $s=0$
and the values of $p$-adic Euler $\ell$-function at positives integers.

\section{Preliminaries}\label{p-adic Euler}

Generalized Euler numbers play fundamental roles in various branches of mathematics
including combinatorics, number theory, special functions and analysis (cf. \cite{Ay,Be,Iwa,Ko,Wa2}).
Now, we present some of the fundamental properties of generalized Euler numbers which are need in the
later sections.

By (\ref{def-E}), (\ref{def-E-pow}), (\ref{gen-E-numb}) and (\ref{gen-E-numb-d}), the generalized Euler numbers
satisfy the relation
\begin{equation}\label{l-e-n}
E_{k,\chi}=f^k\sum_{a=1}^{f}(-1)^a\chi(a)E_{k}\left(\frac{a}f\right).
\end{equation}
In particular, $E_{0,\chi}=\sum_{a=1}^{f}(-1)^a\chi(a)$ for all $\chi.$
From (\ref{power-sum}) and (\ref{ell-ft}) we can deduce the formula
\begin{equation}\label{sum-ell}
2\sum_{n=1}^\infty(-1)^n\chi(n)e^{-nt}=\sum_{k=0}^\infty\ell_E(-k,\chi)\frac{(-t)^k}{k!}
\end{equation}
by considering the vertical line integral
$\frac1{2\pi i}\int_{2-i\infty}^{2+i\infty}t^{-s}\Gamma(s)\ell_E(s,\chi)ds$
and moving the path integration to Re$(s)=-\infty$ (see \cite[Eq.\,(3.9)]{MR}).
The left hand side of (\ref{sum-ell}) can be simplified as follows:
\begin{equation}\label{sum-eu}
\begin{aligned}
2\sum_{n=1}^\infty(-1)^n\chi(n)e^{-nt}&=2\sum_{a=1}^{f}(-1)^a\chi(a)e^{-at}\sum_{k=0}^\infty(-1)^ke^{-fkt}\\
&=\sum_{k=0}^\infty\left(f^k\sum_{a=1}^f(-1)^a\chi(a)E_k\left(1-\frac af\right)\right)\frac{t^k}{k!}.
\end{aligned}
\end{equation}
Using the identity $E_{k}(1-a/f)=(-1)^kE_k(a/f)$ and (\ref{sum-eu}), we observe that
\begin{equation}\label{ell-eu-ne}
\ell_E(-k,\chi)=f^k\sum_{a=1}^{f}(-1)^a\chi(a)E_{k}\left(\frac{a}f\right).
\end{equation}
From (\ref{l-e-n}) we can rewrite (\ref{ell-eu-ne}) as
\begin{equation}\label{ell-ft-val}
\ell_E(-n,\chi)=E_{n,\chi}
\end{equation}
for $n\geq0$ (cf. \cite{TK1,KKJR}).
If $\chi=\chi^0,$ we have
\begin{equation}\label{z-ft-val}
\zeta_E(-n)=E_n
\end{equation}
when $n\geq1.$ Further, if $n=0,$ $\zeta_E(0)=-2(1-1+1-1+\cdots)=-E_0$
(see \cite[p.\,189, (2)]{Ay}).
It is clear from (\ref{gen-E-numb}) that
$F_\chi(-t)=-\chi(-1)F_\chi(t),$ if $\chi\neq\chi^0,$ the trivial character.
Hence
\begin{equation}\label{sym-g-E}
(-1)^{n+1}E_{n,\chi}=\chi(-1)E_{n,\chi},\quad n\geq0.
\end{equation}
In particular we obtain
\begin{equation}\label{sym-E}
E_{n,\chi}=0\quad\text{if }\chi\neq\chi^0,\; n\not\equiv\delta_\chi\pmod2,
\end{equation}
where $\delta_\chi=0$ if $\chi(-1)=-1$ and $\delta_\chi=1$ if $\chi(-1)=1.$
However, since
$F_{\chi^0}(t)=-2e^t/(e^t+1),$
we easily see that
$E_{n,\chi^0}=E_n$ if $n$ odd and 0 if $n$ even,
and $E_{0,\chi^0}=-E_0=-1.$
From (\ref{ell-ft-val}) and (\ref{sym-g-E}), we may deduce that
\begin{equation}\label{sym-eq-l}
\ell_{E}(-n,\chi)=(-1)^{n+1}\chi(-1)\ell_{E}(-n,\chi),\quad n\geq0.
\end{equation}
For $n\geq0$ therefore, $\ell_{E}(-n,\chi)\neq0$ if and only if $\chi(-1)=(-1)^{n+1}.$
That is, $\ell_{E}(-n,\chi)\neq0$ if and only if $\chi$ and $n+1$ have the same parity.

Let $\mathbb Q_p(\chi)$ denote the field generated over $\mathbb Q_p$ by $\chi(a),a\in\mathbb Z$
(in an algebraic closure of $\mathbb Q_p$).
$\mathbb Q_p(\chi)$ is a locally compact topological field containing $\mathbb Q(\chi)$ as a dense subfield.
We can state the following lemma.

\begin{lemma}\label{pro-ge-E}
If $n\in\mathbb Z, n\geq0,$ then there exist a Witt's formula of $E_{n,\chi}$
in $\mathbb Q_p(\chi)$ such that
$$E_{n,\chi}=\lim_{N\to\infty}\sum_{a=1}^{fp^N}(-1)^a\chi(a)a^n.$$
Herein as usual we set $\chi(a)=0$ if $a$ is not prime to the conductor $f.$
\end{lemma}
\begin{proof}
For $t\in\overline{\mathbb Q}_p$ with $|t|_p<|p|_p^{\frac1{p-1}},$ we have
$$\lim_{N\to\infty}\sum_{b=0}^{p^N-1}(-1)^be^{bt}=\lim_{N\to\infty}\frac{1+e^{p^Nt}}{1+e^t}$$
(see \cite{Ko}).
As $\lim_{N\to\infty}e^{p^Nt}=1,$ we obtain
\begin{equation}\label{Eu-witt-gen}
\lim_{N\to\infty}\sum_{b=0}^{p^N-1}(-1)^be^{bt}=\frac{2}{1+e^t},
\end{equation}
so
\begin{equation}\label{Eu-witt}
\begin{aligned}
\lim_{N\to\infty}\sum_{a=1}^{fp^N}(-1)^a\chi(a)e^{ta}
&=\lim_{N\to\infty}\sum_{a=1}^f(-1)^a\chi(a)\sum_{b=0}^{p^N-1}(-1)^b(e^{tf})^be^{at} \\
&=F_\chi(t),
\end{aligned}
\end{equation}
where $\chi$ is a primitive Dirichlet character with an odd conductor $f=f_\chi.$
Using (\ref{gen-E-numb}) and comparing the coefficients of $t^n/n!$ on both sides of (\ref{Eu-witt}),
we obtain the result.
\end{proof}

\begin{corollary}\label{pro-ge-E-coro}
In $\mathbb Q_p(\chi),$
$$(1-\chi(p)p^n)E_{n,\chi}
=\lim_{N\to\infty}\sum_{\substack{a=1\\ p\nmid a}}^{fp^N}(-1)^a\chi(a)a^n.$$
\end{corollary}

We define the difference operator $\Delta_c$ by the forward difference operator
\begin{equation}\label{diff}
\Delta_ca_n=a_{n+c}-a_n.
\end{equation}
Repeated application of this operator can be expressed in the form
\begin{equation}\label{diff-o}
\Delta_c^k a_n=\sum_{j=0}^k\binom kj (-1)^{k-j}a_{n+jc}
\end{equation}
for some positive $k\in\mathbb Z$ (see \cite{Yo}).
An application of (\ref{diff-o}) to the sequence $\{E_{n,\chi}\}$ yields
\begin{equation}\label{diff-Euler}
\Delta_{p-1}^kE_{n,\chi}=\lim_{N\to\infty}\sum_{a=1}^{fp^N}(-1)^a\chi(a)a^n(a^{p-1}-1)^k,
\end{equation}
where $n,k\geq1.$ Therefore, we obtain the congruence in the field $\mathbb Q_p(\chi):$
\begin{equation}\label{diff-Euler-cong}
\Delta_{p-1}^kE_{n,\chi}\equiv0\pmod{p^k}
\end{equation}
for all natural numbers $n$ such that $n\geq k.$
This is an analogue to Kummer's congruences for the ordinary Bernoulli numbers $B_n.$
Therefore we obtain the following theorem.

\begin{theorem}\label{diff-o-thm}
If $n$ be the natural numbers such that $n\geq k.$ Then
$$\Delta_{p-1}^kE_{n,\chi}\equiv0\pmod{p^{k}}.$$
\end{theorem}

\begin{remark}\label{diff-o-remark}
Shiratani \cite[Eq. (10)]{Sh} has noted that
$$\Delta_{p-1}^k\frac1n B_{n}\equiv0\pmod{p^{k}}$$
for all natural numbers $n$ such that $n\geq k+1.$
This congruence is the most well-known formula of Kummer (cf. \cite{Ha,Yo}).
\end{remark}

\section{Construction of $p$-adic Euler $\ell$-functions}\label{con-eu-ell}

In this section, we investigate several interesting
properties of $p$-adic Euler $\ell$-functions interpolates the generalized Euler numbers
in methods similar to \cite[$\S$3]{Iwa} and \cite[$\S$5.3]{Wa2}.

For each $a\in\mathbb Z_p$ with $p\nmid a,$ $a$ can be uniquely written in the form
\begin{equation}\label{ome}
\langle a\rangle=\omega^{-1}(a)a,
\end{equation}
where $\omega$ is the Teichm\"uller character.
Then we have $\langle a\rangle\equiv 1\pmod{p\mathbb Z_p}.$
For $p>2,$ $\lim_{n\to\infty}a^{p^n}=\omega(a)$ (see \cite{Wa2}).

Let $\chi$ be the Dirichlet character with an odd conductor $f=f_\chi.$ For $n\geq1,$ we define $\chi_n$
to be the primitive character associated with the character
$$\chi_n:(\mathbb Z/\text{l.c.m.}(f,p)\mathbb Z)^\times\rightarrow\mathbb C^\times$$
defined by $\chi_n(a)=\chi(a)\omega^{-n}(a).$
Let $\Q(\chi)$ be the field generated over $\Q$ by the values $\chi(a),a\in\mathbb Z.$
We define a sequence of elements $\epsilon_{n,\chi},n\geq0,$ in $\Q(\chi)$ by
\begin{equation}\label{epsi-eu}
\epsilon_{n,\chi}=(1-\chi_n(p)p^n)E_{n,\chi_n},
\end{equation}
where $E_{n,\chi_n}$ is the generalized Euler number defined in (\ref{gen-E-numb-d}).
Note that $\chi_n(a)$ is in $\Q(\chi)$ for any $n\geq0$ and $a\in\mathbb Z.$

We put
$$\mathcal D=\left\{s\in\C\,:\, |s|_p<|p|_p^{-(p-2)/(p-1)}\right\}.$$
Now we define an interpolation function $\ell_{p,E}(s,\chi)$ for generalized Euler numbers over
$\mathcal D$ by
\begin{equation}\label{p-l-fts}
\ell_{p,E}(s,\chi)=
\lim_{N\to\infty}\sum_{\substack{a=1\\ p\nmid a}}^{fp^N}(-1)^a\chi(a)
\langle a\rangle^{1-s}
\end{equation}
(cf. \cite{Iwa,KS2,Wa,Wa2}).
The $p$-adic function $\ell_{p,E}(s,\chi)$ will be called the $p$-adic Euler $\ell$-function.
From Lemma \ref{pro-ge-E} and (\ref{ome}) we easily see that
\begin{equation}\label{p-l-n-values}
\begin{aligned}
\ell_{p,E}(1-n,\chi)&=\lim_{N\to\infty}\sum_{\substack{a=1\\ p\nmid a}}^{fp^N}(-1)^a\chi(a)
\langle a\rangle^{n} \\
&=\lim_{N\to\infty}\left(\sum_{a=1}^{fp^N}(-1)^a\chi(a) \langle a\rangle^{n}
-\sum_{a=1}^{fp^{N-1}}(-1)^{ap}\chi(ap)\langle ap\rangle^{n}\right) \\
&=\epsilon_{n,\chi},\quad n\geq1.
\end{aligned}
\end{equation}

\begin{theorem}\label{Kum3}
Let $n,c,k$ be positive integers with $c\equiv0\pmod{p-1},p>2.$ Then
$$\Delta_c^k\epsilon_{n,\chi}\equiv0\pmod{p^{k}\mathbb Z_p[\chi]}.$$
\end{theorem}
\begin{proof} From (\ref{diff-o}) and (\ref{p-l-n-values}), we note that
$$\begin{aligned}
\Delta_c^k\epsilon_{n,\chi}&=\Delta_c^k\ell_{p,E}(1-n,\chi) \\
&=\sum_{j=0}^k\binom kj (-1)^{k-j}\lim_{N\to\infty}\sum_{\substack{a=1\\ p\nmid a}}^{fp^N}(-1)^a\chi(a)
\langle a\rangle^{n+jc} \\
&=\lim_{N\to\infty}\sum_{\substack{a=1\\ p\nmid a}}^{fp^N}(-1)^a\chi(a)
(\langle a\rangle^{c}-1)^k\langle a\rangle^n.
\end{aligned}$$
If $c\equiv0\pmod{p-1},$ then since $\langle a\rangle\in\mathbb Z_p^\times,$ we have
$$(\langle a\rangle^{c}-1)^k\equiv0\pmod{p^{k}\mathbb Z_p}$$
by Euler's Theorem. We thus conclude that
$\Delta_c^k\epsilon_{n,\chi}\equiv0\pmod{p^{k}\mathbb Z_p[\chi]}.$
\end{proof}

\begin{corollary}
Let $n\equiv n^\prime\pmod{p-1}$ and $c\equiv0\pmod{p-1}.$ Then
$$\Delta_c^k\epsilon_{n,\chi}\equiv\Delta_c^k\epsilon_{n^\prime,\chi}\pmod{p^{k+1}\mathbb Z_p[\chi]}.$$
\end{corollary}
\begin{proof}
Note that
$$\Delta_c^k\epsilon_{n,\chi}-\Delta_c^k\epsilon_{n',\chi}
=\lim_{N\to\infty}\sum_{\substack{a=1\\ p\nmid a}}^{fp^N}(-1)^a\chi(a)
(\langle a\rangle^{c}-1)^k(\langle a\rangle^{n-n'}-1)\langle a\rangle^{n'}.$$
If $n-n'\equiv0\pmod{p-1},$ then we have $\langle a\rangle^{n-n'}\equiv1\pmod{p\Z}.$
Thus the result follows from Theorem \ref{Kum3}.
\end{proof}

\begin{lemma}[{\cite[p.\,19, Lemma 1]{Iwa}}]\label{uniq}
If $A(x),B(x)\in\mathbb Q_p(\chi)[[x]],$ convergent in a neighborhood of 0 in $\overline{\mathbb Q}_p$
and $A(\xi_n)=B(\xi_n)$ for a sequence of elements $\xi_n\neq0,n\geq0,$ in $\overline{\mathbb Q}_p$ such that
$\lim_{n\to0}\xi_n=0,$ then $A(x)=B(x).$
\end{lemma}

Let
\begin{equation}\label{c-e}
c_n=\sum_{i=0}^n\binom ni(-1)^{n-i}\epsilon_{i,\chi},\quad n\geq0.
\end{equation}

\begin{lemma}[cf. {\cite[p.\,26, Lemma 4]{Iwa}}]\label{c_n}
$$|c_n|_p\leq |p^n|_p,\quad n\geq0.$$
\end{lemma}

\begin{proof}
From the definition of $\epsilon_{i,\chi}$ and (\ref{p-l-n-values}), we get
$$\begin{aligned}
c_n&=\sum_{i=0}^n\binom ni(-1)^{n-i}
\lim_{N\to\infty}\sum_{\substack{a=1\\ p\nmid a}}^{fp^N}(-1)^a\chi(a)\langle a\rangle^{i}\\
&=\lim_{N\to\infty}\sum_{\substack{a=1\\ p\nmid a}}^{fp^N}(-1)^a\chi(a)(\langle a\rangle-1)^n.\\
\end{aligned}$$
Since $\langle a\rangle\equiv1\pmod p$ then $(\langle a\rangle-1)^n\equiv0\pmod{p^n},$ hence
we can conclude
$$\begin{aligned}
|c_n|_p&=\lim_{N\to\infty}
\biggl|\sum_{\substack{a=1\\ p\nmid a}}^{fp^N}(-1)^a\chi(a)(\langle a\rangle-1)^n\biggl|_p \\
&=\lim_{N\to\infty}|p^n\theta_n(N)|_p\leq|p^n|_p,
\end{aligned}$$
because
$$\sum_{\substack{a=1\\ p\nmid a}}^{fp^N}(-1)^a\chi(a)(\langle a\rangle-1)^n=p^n \theta_n(N)$$
for some $\theta_n(N)$ with $|\theta_n(N)|_p\leq1.$
This is the desired conclusion.
\end{proof}

We now apply Theorem 1 in \cite[p.\,22]{Iwa} for the above sequences $\epsilon_{n,\chi}$ and $c_n,\;n\geq0,$
in $\Q(\chi)$ and for
$$r=|p|_p<|p|_p^{1/(p-1)}.$$
Thus we show that
there exists such $A_\chi(x)\in\Q(\chi)[[x]]$ convergent for
$|\xi|_p<|p|_p^{1/(p-1)}|p|_p^{-1}=|p|_p^{-(p-2)/(p-1)}$ which takes the prescribed values at the non-negative integers,
\begin{equation}\label{ell-A-e}
A_\chi(n)=\epsilon_{n,\chi}.
\end{equation}
Let
\begin{equation}\label{ell-A}
\ell_{p,E}(s,\chi)=A_\chi(1-s)
\end{equation}
with the $A_\chi(x)$ mentioned above.
The uniqueness of $\ell_{p,E}(s,\chi)$ is a consequence of Lemma \ref{uniq}.
Then we have the following theorem.

\begin{theorem}\label{Iwa-e-ell}
Let $\chi$ be a Dirichlet character with an odd conductor $f=f_\chi.$ Then
there exists a $p$-adic analytic function
$$\ell_{p,E}(s,\chi)=\sum_{n=0}^\infty(-1)^n a_n(s-1)^n,\quad a_n\in\Q(\chi)$$
defined on $\mathcal D$
such that
$$\ell_{p,E}(1-n,\chi)=\epsilon_{n,\chi},\quad n\geq1.$$
Moreover, the $p$-adic Euler $\ell$-function $\ell_{p,E}(s,\chi)$ interpolates the numbers
$(1-\chi_n(p)p^n)\ell_E(-n,\chi_n)$ for $n\geq0.$
\end{theorem}

\begin{remark} Since $\chi_n=\chi$ whenever $n\equiv0\pmod{p-1},$ we have
$$\ell_{p,E}(1-n,\chi)=0\quad\text{if }\chi\neq\chi^0,\;n\not\equiv\delta_\chi\!\!\!\!\pmod2,\;n\equiv0\!\!\!\!\pmod{p-1}.$$
\end{remark}

\begin{corollary}\label{Iwa-e-ell-co}
Suppose $\chi\neq1$ and $p^2\nmid f_\chi.$ Then
$$\ell_{p,E}(s,\chi)=a_0-a_1(s-1)+a_2(s-1)^2-\cdots$$
with $|a_0|_p\leq1$ and with $p\mid a_n$ for all $n\geq1.$
\end{corollary}
\begin{proof}
By using (\ref{p-l-fts}) and the formula
$$\langle a\rangle^{1-s}=\exp((1-s)\log_p\langle a\rangle)=\sum_{n=0}^\infty\frac1{n!}(1-s)^n(\log_p\langle a\rangle)^n,$$
where $\exp$ (resp. $\log_p$) is the $p$-adic exponential (resp. logarithm) function (see \cite{Iwa}), we have
$$\ell_{p,E}(s,\chi)=\sum_{n=0}^\infty(-1)^n(s-1)^n\lim_{N\to\infty}\sum_{\substack{a=1\\ p\nmid a}}^{fp^N}(-1)^a\chi(a)
\frac{(\log_p \langle a\rangle)^n}{n!},\quad n\geq0.$$
From Theorem \ref{Iwa-e-ell}, we find that
$$a_n=\lim_{N\to\infty}\sum_{\substack{a=1\\ p\nmid a}}^{fp^N}(-1)^a\chi(a)
\frac{(\log_p \langle a\rangle)^n}{n!},\quad n\geq0.$$
If $n\geq1,$ then ${(\log_p \langle a\rangle)^n}/{n!}\equiv0\pmod{p}.$ Thus we obtain
$$p\mid a_n\quad\text{for }n\geq 1.$$
By Lemma \ref{c_n} $|a_0|_p=|c_0|_p\leq1$ is obvious.
\end{proof}

By Corollary \ref{Iwa-e-ell-co}, we have
\begin{equation}
\ell_{p,E}(s,\omega^m)=\ell_{p,E}(s,\omega^n)
\end{equation}
if $m\equiv n\not\equiv0\pmod{p-1},$
hence we have the following which can be proved in the same way as Corollary 5.14 in \cite{Wa2}.

\begin{corollary}\label{Iwa-e-ell-co-2}
If $m$ and $n$ are positive even integers with $m\equiv n\pmod{(p-1)p^k}$ and $n\not\equiv0\pmod{p},$ then
$$(1-p^m)E_m\equiv(1-p^n)E_n\pmod{p^{k+1}}.$$
\end{corollary}

Let $\chi$ be the Dirichlet character with an odd conductor $f=f_\chi\in\mathbb N.$
Let $F$ be a positive integer multiple of $p$ odd and $f.$ Then by (\ref{power-sum}), we have
\begin{equation}\label{rabbe}
F_{\chi}(t)=2\sum_{m=0}^\infty(-1)^m\chi(m)e^{mt}=2\sum_{a=1}^{F}(-1)^a\chi(a)\frac{(e^{\frac aF})^{Ft}}{e^{Ft}+1}.
\end{equation}
Therefore, by (\ref{def-E}), (\ref{def-E-pow}) and (\ref{gen-E-numb-d}), we obtain the following
\begin{equation}\label{rabbe-lem}
E_{n,\chi}=F^n\sum_{a=1}^{F}(-1)^a\chi(a)E_{n}\left(\frac{a}{F}\right).
\end{equation}
If $\chi_n(p)\neq0,$ then $(p,f_{\chi_n})=1,$ so that $F/p$ is a multiple of $f_{\chi_n}.$
From (\ref{rabbe-lem}), we derive
\begin{equation}\label{rabbe-ell}
\begin{aligned}
\chi_n(p)p^nE_{n,\chi_n}
&=\chi_n(p)p^n\left(\frac Fp\right)^n\sum_{a=1}^{F/p}(-1)^a\chi_n(a)E_{n}\left(\frac{a}{F/p}\right)\\
&=F^n\sum_{\substack{a=1\\p\mid a}}^F(-1)^a\chi_n(a)E_{n}\left(\frac{a}{F}\right).
\end{aligned}
\end{equation}
Thus by (\ref{epsi-eu}), (\ref{rabbe-lem}) and (\ref{rabbe-ell}), we have
\begin{equation}\label{ell-le}
\begin{aligned}
\epsilon_{n,\chi}
&=F^n\sum_{\substack{a=1\\p\nmid a}}^{F}(-1)^a\chi_n(a)E_{n}\left(\frac{a}{F}\right).
\end{aligned}
\end{equation}
Since $E_n(x)=\sum_{k=0}^n\binom nkx^{n-k}E_k$ and $\chi_n(a)=\chi(a)\omega^{-n}(a),$ by (\ref{ell-le}), we easily see that
\begin{equation}\label{ell-le-ft}
\epsilon_{n,\chi}
=\sum_{\substack{a=1\\p\nmid a}}^{F}(-1)^a\chi(a)\langle a\rangle^n
\sum_{k=0}^\infty\binom nk\left(\frac Fa\right)^kE_k,
\end{equation}
where $\langle a\rangle=\omega^{-1}(a)a.$ From Theorem \ref{Iwa-e-ell} and (\ref{ell-le-ft}), we obviously have
\begin{equation}\label{ne-val}
\begin{aligned}
\ell_{p,E}(-n,\chi)=\sum_{\substack{a=1\\p\nmid a}}^{F}(-1)^a\chi(a)\langle a\rangle^{1+n}
\sum_{k=0}^\infty\binom{1+n}k\left(\frac Fa\right)^kE_k
\end{aligned}
\end{equation}
for $n\geq0$ (cf. \cite{TK1,KKJR,La,Wa}).
Using Theorem \ref{Iwa-e-ell} and (\ref{ne-val}),
our $p$-adic Euler $\ell$-function $\ell_{p,E}(s,\chi)$ can be written in the form

\begin{theorem}\label{Wa-e-ell}
Let $\chi$ be a Dirichlet character with an odd conductor $f=f_\chi,$
and let $F$ be a positive integer multiple of $p$ odd and $f.$ Then
$$\ell_{p,E}(s,\chi)=\sum_{\substack{ a=1 \\  p\nmid a}}^{F}(-1)^a\chi(a) H_p(s,a,F),\quad s\in\Z.$$
Here $H_p(s,a,F)$ is the Washington function \cite[$\S$5.2]{Wa2}. It is defined for
$a\in\mathbb Z_p^\times,$ $s\in \Z$ and $p\mid F$ by
$$H_p(s,a,F)=\langle a\rangle^{1-s}
\sum_{k=0}^\infty\binom {1-s}k \left(\frac Fa\right)^{k}E_{k}.$$
This function is analytic function for $s\in \Z.$
\end{theorem}
\begin{proof}
The left hand side and right the hand side have the same values at the negative integers, which are dense in
the ring of $p$-adic integers $\Z,$
and they are both analytic, hence they coincide.
\end{proof}

Next, we derive some functional equations for the Washington function (cf. \cite{La}).

\begin{proposition}\label{H-p-ft}
If $s\in\mathbb Z_p$ and $a\in\mathbb Z_p^\times,$ then
$$H_p(s,a,F)=H_p(s,F-a,F).$$
\end{proposition}
\begin{proof}
From the definition of $H_p(s,a,F),$ we obtain
$$H_p(1-n,a,F)=\langle a\rangle^{n}
\sum_{k=0}^n\binom nk \left(\frac aF\right)^{-k}E_{k}=\langle a\rangle^{n}E_n\left(\frac{a}{F}\right),$$
where $n\geq1,a\in\mathbb Z_p^\times$ and $p\mid F$ (cf. \cite{Wa}).
One can see easily from this that
$$\begin{aligned}
H_p(1-n,F-a,F)
&=\langle F-a\rangle^n\left(\frac{F}{F-a}\right)^n \sum_{k=0}^n\binom nk\left(\frac{F-a}{F}\right)^{n-k}E_k \\
&=\langle F-a\rangle^n\left(\frac{F}{F-a}\right)^n E_n\left(1-\frac{a}{F}\right) \\
&=\langle a\rangle^n \langle -1\rangle^n \left(\frac{F}{a}\right)^n E_n\left(\frac{a}{F}\right) \\
&=H_p(1-n,a,F),
\end{aligned}$$
since $E_n(1-x)=(-1)^nE_n(x)$
and $\omega^{-1}(F-a)=\omega^{-1}(-a)=-\omega^{-1}(a)$ (cf. \cite{Iwa}).
Using the fact that the set of non-negative integers are dense in $\mathbb Z_p,$ we see that
$$H_p(s,a,F)=H_p(s,F-a,F)$$
for $s\in\mathbb Z_p$ and $a\in\mathbb Z_p^\times.$ This is the desired conclusion.
\end{proof}

\section{Evaluation of $\ell_{p,E}'(0,\chi)$ and $\ell_{p,E}(n,\chi)$ at $n\geq1$}

In finding the value of the derivation of the function $\ell_{p,E}(s,\chi)$ at $s=0$ and
the values of $\ell_{p,E}(n,\chi)$ at $n\geq1$ (cf. \cite{Di2,FG}), we define the function
\begin{equation}\label{Di-ft}
G_{p,E}(x)=\lim_{N\to\infty}\sum_{a=0}^{p^N-1}\left\{(x+a)\log_p(x+a)-(x+a)\right\}(-1)^a
\end{equation}
for $x\in\C$ with $|x|_p>1$ (cf. \cite{TK1}).
This definition is slightly different from the original one due to Diamond (cf. \cite[p.\,326, Definition of $G_p$]{Di}).
Here $\log_p$ is the $p$-adic logarithm function of Iwasawa (see \cite{Iwa}).
Let $|x|_p>1.$ For $a\in\mathbb Z_p$ we have $\left|\frac ax\right|_p<1$ so that
$$\begin{aligned}
&(x+a)\log_p(x+a)-(x+a) \\
&=x\left(1+\frac ax\right)\log_p\left(1+\frac ax\right)+(x+a)\log_p(x)-(x+a) \\
&=a+x\sum_{n=1}^\infty\frac{(-1)^{n+1}}{n(n+1)}\left(\frac ax\right)^{n+1}+(x+a)\log_p(x)-(x+a)
\end{aligned}$$
(cf. \cite{Di}). From (\ref{Eu-witt-gen}), we easily obtain that
\begin{equation}\label{Eu-witt-n}
\lim_{N\to\infty}\sum_{a=0}^{p^N-1}a^n(-1)^a=E_n,
\end{equation}
so
\begin{equation}\label{Di-ft-Eu}
\begin{aligned}
G_{p,E}(x)
&=\left(x-\frac12\right)\log_p(x)-x-\sum_{n=1}^\infty\frac{1}{n(n+1)}\frac1{x^n}E_{n+1} \\
&=\left(x-\frac12\right)\log_p(x)-x-\sum_{n=1}^\infty\frac{1}{n(n+1)}\frac1{x^n}\zeta_{E}(-n-1),
\end{aligned}
\end{equation}
where we use (\ref{z-ft-val}) and
the fact that $E_{n+1}=0$ if $n$ is odd. This formula arises from the asymptotic expansion of the classical
complex log gamma function.
We put
$$D=\frac{\text{d}}{\text{d}x}.$$
Furthermore, if $D^m$ denotes the $m$-th derivative, by (\ref{Di-ft-Eu}), we have
\begin{equation}\label{h-Di-ft-Eu}
D^{m}G_{p,E}(x)=(-1)^{m}(m-2)!\sum_{k=0}^\infty\binom{-m+1}k\frac1{x^{m+k-1}}E_k,\quad m\geq2.
\end{equation}

\begin{proposition}\label{Da-gamma}
Let $p$ be a fixed odd prime number. Then
\begin{enumerate}
\item $G_{p,E}(1-x)+G_{p,E}(x)=0$ $(x\in\mathbb C_p-\mathbb Z_p).$
\item $G_{p,E}(x)-G_{p,E}(-x)=2x(\log_p(x)-1)$ $(x\in\mathbb C_p-\mathbb Z_p).$
\item $G_{p,E}(1+x)+G_{p,E}(x)=2x(\log_p(x)-1)$ $(x\in\mathbb C_p-\mathbb Z_p).$
\end{enumerate}
\end{proposition}
\begin{proof}
Upon expanding either side of Proposition \ref{H-p-ft} in powers of $s$ and equating the coefficients of $s,$ we get
$$\begin{aligned}
&\omega^{-1}(a)\left(\frac F2-a\right)(1+\log_p(F))-\omega^{-1}(a)FG_{p,E}\left(\frac aF\right) \\
&=\omega^{-1}(F-a)\left(\frac F2-(F-a)\right)(1+\log_p(F))-\omega^{-1}(F-a)FG_{p,E}\left(\frac {F-a}F\right).
\end{aligned}$$
This yields
$$G_{p,E}\left(\frac aF\right)+G_{p,E}\left(1-\frac {a}F\right)=0.$$
This is true for any positive integer $F$ divisible by $p,$ and any $p$-unit $a,$ thus established (1).
Using $\log_p(-x)=\log_p(x),$ (2) is straightforward consequences of the the power series expansion.
Formula (3) follows form (1) and (2).
\end{proof}

Now we prove a formula for $\ell_{p,E}'(0,\chi)$ which is analogous to a classical formula of $L_{p}'(0,\chi)$
(see \cite{Di,FG}).

\begin{theorem}\label{der-ell}
Let $\chi$ be a primitive Dirichlet character, and let $F$ be a positive odd integral multiple of $p$ and $f_\chi.$
Then
$$\ell_{p,E}'(0,\chi)=F\sum_{\substack{ a=1\\  p\nmid a}}^{F}(-1)^{a+1} \chi_1(a)G_{p,E}\left(\frac aF\right)-
(1+\log_p(F))\ell_{p,E}(0,\chi).$$
\end{theorem}
\begin{proof}
We have the expansions (cf. \cite{Fo}):
$$\begin{aligned}
&\langle a\rangle^{1-s}=\langle a\rangle(1-s\log_p\langle a\rangle+\cdots \\
&\binom{1-s}n=\frac{(-1)^{n+1}}{n(n-1)}s+\cdots,
\end{aligned}$$
provided $n\geq2.$ From these expansions and Theorem \ref{Wa-e-ell}, we find that the coefficient of $s$ in $\ell_{p,E}(s,\chi)$ is
$$\sum_{\substack{ a=1\\  p\nmid a}}^{F}(-1)^{a} \chi_1(a)
\left(\frac F2-\left(a-\frac F2\right)\log_p\langle a\rangle+F\sum_{n=1}^\infty\frac{1}{n(n+1)}
\left(\frac Fa\right)^{n}E_{n+1}\right).$$
Furthermore, from (\ref{Di-ft-Eu}), we have
$$G_{p,E}\left(\frac aF\right)=\left(\frac aF-\frac12\right)\log_p\left(\frac aF\right)-\frac aF-
\sum_{n=1}^\infty\frac{1}{n(n+1)}\left(\frac Fa\right)^nE_{n+1}.$$
Since the value of $\ell_{p,E}'(0,\chi)$ is the coefficient of $s$ in the expansion of $\ell_{p,E}(s,\chi)$ about $s=0,$
by evaluating the sum
$$\begin{aligned}
\sum_{\substack{ a=1\\  p\nmid a}}^{F}(-1)^{a} \chi_1(a)\left(a-\frac F2\right)
&=\sum_{a=1}^{F}(-1)^{a} \chi_1(a)\left(a-\frac F2\right) \\
&\quad-\chi_1(p)p\sum_{a=1}^{F/p}(-1)^{a} \chi_1(a)\left(a-\frac{\frac Fp}2\right) \\
&=(1-\chi_1(p)p)E_{1,\chi_1} \\
&=\ell_{p,E}(0,\chi),
\end{aligned}$$
we obtain the result.
\end{proof}

We can now obtain  the values of $p$-adic Euler $\ell$-functions at positives integers (cf. \cite{Di2}).

\begin{theorem}\label{h-der-ell}
Let $\chi$ be a primitive Dirichlet character with an odd conductor $f=f_\chi.$
Then
$$\ell_{p,E}(n,\chi_{n-1})=\frac{(-pf)^{-n+1}}{(n-2)!}\sum_{\substack{ a=1\\  p\nmid a}}^{fp}(-1)^{a}\chi(a)
\left(D^{n}G_{p,E}\right)\left(\frac a{pf}\right),\quad n\geq2.$$
In particular, $\ell_{p,E}(1,\chi)=(1-\chi(p))E_{0,\chi}.$
\end{theorem}
\begin{proof}
Note that
\begin{equation}\label{ra-in-exp}
\frac1{(x+a)^n}=\sum_{k=0}^\infty\binom{-n}{k}x^{-n-k}a^k.
\end{equation}
Applying (\ref{Eu-witt-n}), (\ref{h-Di-ft-Eu}) and (\ref{ra-in-exp}), we have
\begin{equation}\label{h-dir-gamma}
D^{n}G_{p,E}(x)=(-1)^{n}(n-2)!\lim_{N\to\infty}\sum_{a=0}^{p^N-1}(-1)^a(x+a)^{-n+1},\quad n\geq2.
\end{equation}
Furthermore, if $n\geq2$ and $\chi$ is a primitive Dirichlet character with an odd conductor $f=f_\chi,$
then by (\ref{p-l-fts}), we obtain
$$\begin{aligned}
\ell_{p,E}(n,\chi_{n-1})
&=\lim_{N\to\infty}\sum_{\substack{a=1\\ p\nmid a}}^{fp^N}(-1)^a\chi(a)
\langle a\rangle^{-n+1} \\
&=\sum_{\substack{a=1\\ p\nmid a}}^{fp}(-1)^a\chi(a)
\lim_{N\to\infty}\sum_{b=0}^{p^{N-1}-1}(-1)^b(a+pfb)^{-n+1}.
\end{aligned}$$
We combine this formula with (\ref{h-dir-gamma}) to obtain our result.
\end{proof}

\section{Further remarks and observations}

The Euler measure is defined for each positive integer $n$ by
(see \cite[Section 4]{TK2})
\begin{equation}\label{Eu-mea}
\mu_{n,E}(a+p^N\Z)=(-1)^ap^{nN}E_n\left(\frac a{p^N}\right),
\end{equation}
where $E_n(x)$ is the Euler polynomials.
Note that $\mu_{n,E}(\Z)=E_n$ and $\mu_{n,E}(p\Z)=p^nE_n.$
We deduce that $\mu_{n,E}(\Z^\times)=(1-p^n)E_n.$
Let $UD(\mathbb Z_p)$ be the set of uniformly differential function on $\mathbb Z_p.$
When $f\in UD(\Z),$ (\ref{Eu-mea}) allows us to define
\begin{equation}\label{Eu-mea-int}
\int_{\Z}f(x)\mu_{n,E}(x)=\lim_{N\to\infty}\sum_{a=0}^{p^N-1}f(a)\mu_{n,E}(a+p^N\Z).
\end{equation}
Hence,
the Euler measures are all related to ``fermionic'' measure $\mu_{-1}(a+p^N\Z)=(-1)^a$ by the property
\begin{equation}\label{Eu-mea-Haar}
\mu_{n,E}(\Z)=\int_{\Z}d\mu_{n,E}(x)=\int_{\Z}x^nd\mu_{-1}(x)=E_n;
\end{equation}
\begin{equation}\label{Eu-pol-mea-Haar}
\sum_{i=0}^n\binom nix^{n-i}\int_{\Z}d\mu_{i,E}(y)=\int_{\Z}(x+y)^nd\mu_{-1}(y)=E_n(x).
\end{equation}
(see \cite{MSK,TK1,TK2}).
This is analogous to the relation between $dx^n$ and $dx$ for $n$ a non-negative integer.
Putting these observations together, we have
\begin{equation}\label{Eu-mea-Mazur}
\int_{\Z^\times}x^nd\mu_{-1}(x)=\mu_{n,E}(\Z^\times)=(1-p^n)E_n,
\end{equation}
where $\Z^\times=\Z\setminus p\Z.$
As we saw in (\ref{z-ft-val}), the right hand side can be interpreted as
\begin{equation}\label{Eu-mea-ell}
(1-p^n)\zeta_{E}(-n)
\end{equation}
which can be extended to give the $p$-adic $\zeta_E$-function (see Section \ref{con-eu-ell}).
Thus, (\ref{Eu-mea-Mazur}) should be viewed
as expressing the $p$-adic $\zeta_E$-function as a kind of Mellin transform (cf. \cite{Iwa,Wa}).
In fact, (\ref{Eu-mea-Mazur}) immediately gives the $p$-adic continuation of the $\zeta_E$-function,
as well as the Kummer congruences. Indeed,
let $m\equiv n\pmod{(p-1)p^N}$ and $p\nmid x.$ From the little Fermat theorem it follows that $x^m\equiv
x^{n}\pmod{p^{N+1}}.$
We have
$$|x^{m}-x^{n}|_p<\frac1{p^{N+1}}\quad\text{for }x\in\Z^\times.$$
Thus
$$\biggl|\int_{\Z^\times}x^md\mu_{-1}(x)-\int_{\Z^\times}x^{n}d\mu_{-1}(x)\biggl|_p\leq\frac1{p^{N+1}}$$
(cf. \cite{KS2}).
We have therefore proved the following theorem (see Corollary \ref{Iwa-e-ell-co-2}).

\begin{proposition}\label{kummer}
If $(p-1)\nmid n$ and $m\equiv n\pmod{(p-1)p^N},$ then
$$(1-p^m)E_m\equiv (1-p^{n})E_{n}\pmod{p^{N+1}}.$$
\end{proposition}

\begin{corollary}\label{kummer2}
If $s$ is a non-negative integer and $N,n$ are positive integers respectively such that $(n,p)=1,N>0,$ then
$$E_{np^N+s}\equiv E_{np^{N-1}+s}\pmod{p^{N}}.$$
\end{corollary}
\begin{proof}
For $x\in\Z,$ we have $x^p\equiv x\pmod{p}$ (Fermat's little theorem).
We see by induction that
$x^{p^N}\equiv x^{p^{N-1}}\equiv{p^N},$
and $\sup_{x\in\Z}|x^{p^N}-x^{p^{N-1}}|_p\leq p^{-N}$ for $N\geq1$ (see \cite[Lemma 2]{Ma}).
Let $(r,p)=1$ and let $s\geq0.$ Then we obtain
$$\biggl|\int_{\Z}x^{s}\left[(x^r)^{p^N}-(x^r)^{p^{N-1}}\right]d\mu_{-1}(x)\biggl|_p \\
\leq p^{-N},$$
which yields the result.
\end{proof}

From (\ref{Eu-mea-Mazur}), we can see that
\begin{equation}\label{von}
|E_n|_p=|1/(1-p^n)|_p\biggl|\int_{\Z^\times}x^nd\mu_{-1}(x)\biggl|_p\leq1,
\end{equation}
because the factor before the integral is coprime to $p$ and therefore their $p$-adic absolute values are 1,
and $|\mu_{-1}(U)|_p\leq1$ for all compact open subsets $U\subset\Z^\times.$
Then we obtain

\begin{proposition}\label{von-Euler}
If $n\geq0,$ then $|E_n|_p\leq1.$
\end{proposition}

Let $d$ be a fixed positive integer.
Let $X=\varprojlim_N(\mathbb Z/dp^N\mathbb Z),$ where the map from
$\mathbb Z/dp^M\mathbb Z$ to $\mathbb Z/dp^N\mathbb Z$ for $M\geq N$
is a reduction mod \,$dp^N.$
Let $a+dp^N\mathbb Z_p=\{x\in X\mid x\equiv a\pmod{dp^N}\}$ and let
\begin{equation}\label{X^*}
X^*=\bigcup_{\substack{0<a<dp\\ (a,p)=1}} (a+dp^N\mathbb Z_p).
\end{equation}
The generalized Euler numbers $E_{n,\chi}$ can be represented by using fermionic expression of $p$-adic integral on $\Z$
as follows \cite{TK2}:
\begin{equation}\label{int-en-ep}
\begin{aligned}
\int_{X}\chi(x)d\mu_{n,E}(x)&=\int_{X}\chi(x)x^nd\mu_{-1}(x)=E_{n,\chi};\\
\int_{pX}\chi(x)d\mu_{n,E}(x)&=\int_{pX}\chi(x)x^nd\mu_{-1}(x)=p^n\chi(p)E_{n,\chi};\\
\int_{X^*}\chi(x)d\mu_{n,E}(x)&=\int_{X\setminus pX}\chi(x)x^nd\mu_{-1}(x)=(1-p^n\chi(p))E_{n,\chi}.
\end{aligned}
\end{equation}
Using (\ref{Eu-pol-mea-Haar}) and (\ref{int-en-ep}), it is not difficult to show that
\begin{equation}\label{dis-rel}
\begin{aligned}
E_{n,\chi}&=\int_{X}\chi(x)x^n d\mu_{-1}(x) \\
&=f^{n-1}\sum_{a=1}^{f}(-1)^a\chi(a)\int_{\mathbb Z_p}\left(\frac {a}f+x\right)^nd\mu_{-1}(x) \\
&=f^{n-1}\sum_{a=1}^{f}(-1)^a\chi(a)\sum_{i=0}^n\binom ni\left(\frac {a}f\right)^{n-i}\int_{\mathbb Z_p}x^id\mu_{-1}(x) \\
&=f^{n-1}\sum_{a=1}^{f}(-1)^a\chi(a)E_n\left(\frac {a}f\right),
\end{aligned}
\end{equation}
since $E_n(x)=\sum_{i=0}^n\binom ni x^{n-i} E_i$ (see (\ref{l-e-n})).
From (\ref{p-l-fts}) and (\ref{int-en-ep}),
the $p$-adic Euler $\ell$-function for a Dirichlet character $\chi$ can be defined by setting
\begin{equation}\label{int-ell-ft}
\ell_{p,E}(s,\chi)=\int_{X^*}\chi(x)\langle x\rangle^{1-s}d\mu_{-1}(x)\quad\text{for }s\in \mathbb Z_p.
\end{equation}
From (\ref{int-en-ep}), we have the following
\begin{equation}\label{dis-rel-val}
\begin{aligned}
\ell_{p,E}(1-n,\chi)&=\int_{X^*}\langle x\rangle^n\chi(x)d\mu_{-1}(x) \\
&=\int_{X^*}x^n\chi_n(x)d\mu_{-1}(x) \\
&=(1-p^n\chi_n(p))E_{n,\chi_n}.
\end{aligned}
\end{equation}
Using the above expression for $\ell_E(-n,\chi)=E_{n,\chi},n\geq0,$ we conclude that
$$\ell_{p,E}(1-n,\chi)=(1-p^n\chi_n(p))\ell_E(-n,\chi_n).$$
This identity can be used to prove the following theorem (see Theorem \ref{Iwa-e-ell}).

\begin{theorem}
There exists a unique $p$-adic continuous function $\ell_{p,E}(s,\chi),s\in\mathbb Z_p,$
such that $\ell_{p,E}(1-n,\chi)=(1-p^n\chi_n(p))\ell_E(-n,\chi_n)$ for $n\geq1.$
\end{theorem}

\end{document}